\documentclass[oneside,12pt]{amsart}
\usepackage{eurosym}
\usepackage{amsfonts}
\usepackage{amsmath,amssymb,amsthm,color}
\usepackage{amsopn}
\usepackage{latexsym}
\usepackage{float}
\usepackage{bm}
\usepackage[utf8]{inputenc}
\usepackage{hhline}

\newcommand{\mmi}{\mathfrak i }
\newcommand{\mr}{\mathfrak r }

\newcommand{\mg}{\mathfrak g }

\newcommand{\ms}{\mathfrak s }

\newcommand{\mmu}{\mathfrak u }

\newcommand{\mk}{\mathfrak k }

\newcommand{\mh}{\mathfrak h }

\newcommand{\mmp}{\mathfrak p }

\newcommand{\so}{\mathfrak{so} }
\newcommand{\su}{\mathfrak{su} }

\renewcommand{\sl}{\mathfrak{sl} }
\newcommand{\gl}{\mathfrak{gl} }

\newcommand{\lela}{ g(}
\newcommand{\rira}{)}

\newcommand{\mcZ}{\mathcal Z}

\newcommand{\bs}{\backslash}

\newcommand{\Ric}{\operatorname{Ric}}
\newcommand{\GL}{\operatorname{GL}}
\newcommand{\SL}{\operatorname{SL}}
\newcommand{\SU}{\operatorname{SU}}

\newcommand{\R}{\mathbb R}
  \renewcommand{\H}{\mathbb H}
\newcommand{\Q}{\mathbb Q}
\newcommand{\N}{\mathbb N}
\newcommand{\Z}{\mathbb Z}

\newcommand{\ts}{\theta^\sharp}

\DeclareMathOperator{\ad}{ad}

\DeclareMathOperator{\tr}{tr}

\DeclareMathOperator{\diag}{diag}

\numberwithin{equation}{section}

 \newtheorem{teo}{Theorem}[section]
 \newtheorem{pro}[teo]{Proposition}
 \newtheorem{cor}[teo]{Corollary}
 \newtheorem{lm}[teo]{Lemma}
 \newtheorem{defi}[teo]{Definition}

 \theoremstyle{definition}
 \newtheorem{ex}[teo]{Example}
 \newtheorem{remark}[teo]{Remark}

\newcommand{\nc}{\newcommand}
\nc{\Iso}{\operatorname{Iso}}
\nc{\Id}{\operatorname{Id}}
 \nc{\iso}{\mathfrak{iso}}
 \nc{\sso}{\mathfrak{so}}
\nc{\Ad}{\operatorname{Ad}} 
\nc{\Sym}{\mathrm{Sym}}
 \nc{\pr}{\operatorname{pr}} 
 \nc{\Dera}{\operatorname{Dera}} 
 \nc{\Auto}{\operatorname{Auto}}
 \nc{\noi}{\noindent}
 \nc{\SO}{\operatorname{\rm SO}}

\title{The structure of locally conformally product Lie algebras}

\author{Viviana del Barco}
\address{V.~del Barco: Instituto de Matemática, Estatística e Computação Científica, Universidade Estadual de Campinas,  Rua Sergio Buarque de Holanda, 651, Cidade Universitaria Zeferino Vaz, 13083-859, Campinas, São Paulo, Brazil.}
\email{delbarc@ime.unicamp.br}

\author{Andrei Moroianu}
\address{A.~Moroianu: Université Paris-Saclay, CNRS,  Laboratoire de mathématiques d'Orsay, 91405, Orsay, France, and Institute of Mathematics “Simion Stoilow” of the Romanian Academy, 21 Calea Grivitei, 010702 Bucharest, Romania}
\email{andrei.moroianu@math.cnrs.fr}


\begin{document}

\begin{abstract} A locally conformally product (LCP) structure on a compact conformal manifold is a closed non-exact Weyl connection (i.e.~a linear connection which is locally but not globally the Levi-Civita connection of Riemannian metrics in the conformal class), with reducible holonomy. 
A left-invariant LCP structure on a compact quotient $\Gamma\backslash G$ of a simply connected Riemannian Lie group $(G,g)$ with Lie algebra $\mg$ can be characterized in terms of a closed 1-form $\theta\in\mg^*$ and a non-zero subspace $\mmu\subset \mg$ satisfying some algebraic conditions.  We show that these conditions are equivalent to the fact that $\mg$ is isomorphic to a semidirect product of a non-unimodular Lie algebra acting on an abelian one by a conformal representation. This extends to the general case results from \cite{AdBM24} holding for solvmanifolds. In addition, we construct explicit examples of compact LCP manifolds which are not solvmanifolds.
\end{abstract}

\maketitle

\section{Introduction}

A locally conformally product (LCP) structure on a compact  manifold $M$ is a pair $(c,D)$ where $c$ is a conformal class of Riemannian metrics on $M$ and $D$ is a closed, non-exact Weyl connection with reducible holonomy. 

In \cite{BM2016} it was conjectured that every such connection has to be flat. The conjecture was proved in loc.~cit.~under a further assumption (tameness of $D$). However, soon after, a counterexample was constructed by Matveev and Nicolayevsky \cite{MN2015}. In some sense, this can be considered as the birth certificate of LCP structures. 

The universal cover $\tilde M$ of any LCP manifold $(M,c,D)$ is simply connected, so the lift of $D$ to $M$ is globally the Levi-Civita of a reducible Riemannian metric $h_D$ on $\tilde M$ (uniquely defined up to constant rescaling). However, $h_D$ is incomplete (since there are elements of the fundamental group of $M$ acting by contractions with respect to $h_D$ and without fixed points), so the global de Rham decomposition theorem does not apply. 

The structure of non-flat LCP manifolds (or more precisely of their universal cover) was described in \cite{MN2017} in the analytic setting and by Kourganoff \cite{Kourganoff} in the general (smooth) setting. Their main result can be stated as follows (see \cite[Theorem 1.5]{Kourganoff}): 
{\em The universal cover $(\tilde M,h_D)$ of a compact non-flat LCP manifold $(M,c,D)$ is globally isometric to a Riemannian product $\mathbb{R}^q\times (N, g_N)$, where $\mathbb{R}^q$ ($q\ge 1$) is the flat Euclidean space, and $(N, g_N)$ is an incomplete Riemannian manifold with irreducible holonomy.}

The distribution tangent to $\R^q$ is preserved by the action of the fundamental group on $\tilde M$, so it descends to a rank $q$ distribution on $M$ called the {\em flat distribution}. A Riemannian metric $g\in c$ on $M$ is called {\em adapted} if the Lee form of $D$ with respect to $g$ vanishes on the flat distribution. The relevance of this notion is due to the following observation of Flamencourt \cite[Def. 3.9]{Fl24}: If $g$ is and adapted metric on a compact LCP manifold $(M,c,D)$ and $(M',g')$ is any compact Riemannian manifold, then the Riemannian product $(M,g)\times (M',g')$ also carries an LCP structure. Notice that every LCP manifold carries adapted metrics \cite[Prop. 3.6]{Fl24}, \cite[Thm. 4.4]{MP24}.

In \cite[Prop. 4.4 and Cor. 4.6]{Fl24} Flamencourt constructed large families of examples of LCP manifolds, including in particular the class of OT manifolds introduced by Oeljeklaus and Toma \cite{OT2005}. He also showed \cite[Prop. 3.11]{Fl24} that the homothety factors of the action of the fundamental group of $M$ on the metric $h_D$ are algebraic numbers, a fact which indicates the strong relationship between LCP manifolds and number field theory.

It was noticed by Kasuya \cite{Ka13} that every OT manifold is a solvmanifold. It is thus natural to study left-invariant LCP structures on solvmanifolds. This was done in \cite{AdBM24} by the authors in collaboration with A.~Andrada. First, in \cite[Prop. 2.4]{AdBM24} it is shown that if $M=\Gamma\bs G$ is a compact manifold obtained as the left quotient of a simply connected Lie group $G$ by a co-compact lattice $\Gamma$, then left-invariant LCP structures on $M$ are in one-to-one correspondence with triples $(g,\theta,\mmu)$ where $g$ is a positive definite scalar product on the Lie algebra $\mg$ of $G$, $\theta\in\mg^*$ is a non-zero linear form on $\mg$ vanishing on the derived algebra $\mg'$, and $\mmu\subset \mg$ is a non-zero $\nabla^\theta$-flat subspace (see Definition \ref{defi} for the precise meaning of this notion). Such a triple $(g,\theta,\mmu)$ is called an LCP structure on the Lie algebra $\mg$. An LCP structure is called {\em proper} if $\mmu$ is a proper subspace of $\mg$, and {\em conformally flat} if $\mmu=\mg$. 

Recall that a necessary condition for a Lie group to admit co-compact lattices is that its Lie algebra is unimodular \cite{Mil76}. Thus, in order to understand left-invariant LCP structures on compact quotients of Lie groups, one has to study LCP structures on unimodular Lie algebras. 

In the solvable case, this was done in \cite[Prop. 4.3 and Cor. 5.5]{AdBM24}, where it is shown that every unimodular solvable LCP Lie algebra is obtained as a semi-direct product $\mh\ltimes_\alpha\R^q$ of a non-unimodular solvable Lie algebra $\mh$ acting on $\R^q$ via a conformal representation $\alpha:\mh\to\so(q)\oplus\R\Id_q$,
whose diagonal component is equal, up to the constant coefficient $-\frac1q$, to the trace form of $\mh$. 

The main goal of the present paper is to extend this result to general unimodular Lie algebras. Along the way, we will also obtain the classification of conformally flat LCP structures (i.e. for which the $\nabla^\theta$-flat subspace $\mmu$ is equal to $\mg$). In the unimodular case, they are all isomorphic to $\R$, $\R^2$ or $\su(2)\oplus \R$ (Prop. \ref{pro:cflatunim} below) and in general every conformally flat LCP Lie algebra is the semidirect product of a unimodular one and an abelian Lie algebra (details are given in Thm. \ref{teo:krn} below). 

In Theorem \ref{teo:unim} we show that every proper LCP structure on a unimodular Lie algebra is adapted. The assumption that the LCP structure is adapted is necessary for most of the other results of Section 4, but the previous theorem just says that this holds automatically when the Lie algebra is unimodular.
Our main result is Theorem \ref{teo:uunimadapt} (summarized in Corollary \ref{cor:main}). It states that every Lie algebra carrying an adapted proper LCP structure is obtained as a semi-direct product, like in the solvable case. 

In Section \ref{sec:nonsolv} we give examples of non-solvable simply connected Lie groups admitting co-compact lattices, whose Lie algebras carry LCP structures. The construction is easier when the semi-simple part of the group is of compact type, whereas for non-compact type, some results about division algebras over $\mathbb{Q}$ are needed.

In the last section, we investigate the set of closed 1-forms on a given unimodular Lie algebra which can occur as Lee forms of proper LCP structures and show that this set is finite (Theorem \ref{teo:lee}). In contrast, Proposition \ref{pro:lee} shows that there are infinitely many closed 1-forms occurring as Lee forms of conformally flat LCP structures on each unimodular Lie algebra carrying conformally flat LCP structures.

\smallskip 

{\bf Acknowledgements:} The authors thank Adrián Andrada for useful discussions and Yves Benoist who pointed out the existence of special co-compact lattices in $\SL(d,\R)$ used in Subsection \ref{ex:sld}. V.~dB is grateful to the Laboratoire de Math\'ematiques d'Orsay and to the Université Paris-Saclay for hospitality and partial financial support. V.~dB is partially supported by FAPESP grant 2023/15089-9. A.~M. is partly supported by the PNRR-III-C9-2023-I8 grant CF 149/31.07.2023 {\em Conformal Aspects of Geometry and Dynamics}.
\medskip 

\section{Preliminaries}

Let $G$ be a connected Lie group endowed with a left-invariant metric $g$, and let $\mg$ denote its Lie algebra. Then $g$ defines an inner product on $\mg$, which we will also denote by $g$.

It is well known that left-invariant tensors on $G$ are in one-to-one correspondence with algebraic tensors in $\mg$. In particular, left-invariant differential $k$-forms on $G$ correspond to elements in $\Lambda^k\mg^*$. A $k$-form $\alpha$ is closed (exact) on $G$ if and only if the corresponding element $\alpha\in \Lambda^k\mg^*$ is closed (resp.~exact) with respect to the Chevalley-Eilenberg differential of $\mg$.
In particular, by using Maurer-Cartan equation, $\theta\in\mg^*$ is closed if and only $\theta|_{\mg'}=0$, where $\mg'$ is the derived algebra of $\mg$. 

The trace form $H^\mg$ of a Lie algebra $\mg$ is the 1-form defined as
\begin{equation}
    H^\mg(x)=\tr(\ad_x), \qquad \forall x\in\mg.
\end{equation}
Since $H^\mg([x,y])=\tr\ad_{[x,y]}=0$ for all $x,y\in \mg$, we obtain that $H^\mg$ is closed. By definition, $\mg$ is unimodular if and only if $H^\mg=0$.

The Levi-Civita connection of the Riemannian manifold $(G,g)$ can be viewed as a linear map $\nabla^g:\mg\to\so(\mg)$ which, by Koszul's formula, satisfies
\begin{equation}
\label{eq:lc}\lela\nabla^g_xy,z\rira=\frac12\left(\lela  [x,y],z\rira-\lela[x,z],y\rira-\lela[y,z],x)\right),\qquad \forall\; x,y,z\in \mg.
\end{equation} Here and henceforth $\so(\mg)$ denotes the space of skew-symmetric endomorphisms of $(\mg,g)$.

Consider the conformal class $c$ of the Riemannian metric $g$ on $G$. A Weyl connection on the conformal manifold $(G,c)$ is a torsion-free linear connection $D$ preserving the conformal class $c$. The fundamental theorem of conformal geometry guarantees a one-to-one correspondence between Weyl connections $D$ and 1-forms $\theta$ \cite{We23}; under this correspondence, $\theta$ is called the Lee form of $D$ with respect to $g$. The Weyl structure is called closed (exact) if its Lee form is closed (resp.~exact); one can easily show that these conditions are independent of the fixed metric in the conformal class. We say that a Weyl structure is left-invariant if its Lee form $\theta$ with respect to the left-invariant metric $g$ is itself left-invariant on $G$; in this case, we also denote its value at the identity by $\theta\in \mg^*$.

The conformal analogue of the Koszul formula allows us to write every left-invariant Weyl connection $D$ in terms of its Lee form $\theta\in\mg^*$ and of the metric $g$ as $D=\nabla^\theta$ where 
\begin{equation}
\label{eq:weylc}
\nabla^\theta_xy:=\nabla^g_xy+\theta(x)y+\theta(y)x-g(x,y)\ts\qquad \forall\; x,y\in \mg, 
\end{equation}
and $\nabla^g$ is defined by \eqref{eq:lc}.
Here and at several other places throughout the paper, 
$\ts\in \mg$ denotes the $g$-dual vector of $\theta$. Using \eqref{eq:lc} together with \eqref{eq:weylc}, the Weyl connection $\nabla^\theta$ can be explicitly defined by the following formula:
    \begin{multline}
\label{eq:LCP}
\lela\nabla^\theta_xy,z\rira=\frac12\left(\lela  [x,y],z\rira-\lela[x,z],y\rira-\lela[y,z],x)\right)\\
+\theta(x)\lela y,z\rira+\theta(y)\lela x,z\rira-\theta(z)g(x,y),\qquad \forall\; x,y,z\in \mg.
\end{multline}

Notice that $\nabla^\theta$ can be seen as linear map $\nabla^\theta:\mg\otimes\mg\to \mg$ so that, for each $x\in\mg$, $\nabla_x^\theta:\mg\to\mg$ is an endomorphism of $\mg$.

Given $x\in \mg$, let  $\theta\wedge x$ denote the 
skew-symmetric endomorphism of $\mg$ defined by $(\theta\wedge x)(y):=\theta(y)x-g(x,y)\ts$, for all $y\in\mg$. Then, \eqref{eq:weylc} can also be written
\begin{equation}
\label{eq:weylc1}
\nabla^\theta_x=\nabla^g_x+\theta\wedge x+\theta(x)\Id_\mg\qquad\forall\; x\in \mg,
\end{equation}
showing that 
\begin{equation}
\label{eq:weylc2}
\nabla^\theta_x-\theta(x)\Id_\mg\in\so(\mg)\qquad\forall\; x\in \mg.
\end{equation}

We are now ready to introduce the objects of study of this paper. 

\begin{defi} \label{defi} A locally conformally product (LCP for short) Lie algebra is a quadruple $(\mg,g,\theta,\mmu)$ where $(\mg,g)$ is a metric Lie algebra, $\theta$ is a non-zero closed $1$-form on $\mg^*$, and $\mmu$ is a non-zero $\nabla^\theta$-flat subspace of $(\mg,g)$. The LCP structure $(g,\theta,\mmu)$ is called  {\em conformally flat} if $\mmu=\mg$ and proper if $\mmu\subsetneq \mg$.
\end{defi}

This definition is motivated by \cite[Prop. 2.4]{AdBM24}, where it is shown that if $M=\Gamma\bs G$ is a compact manifold obtained as the left quotient of a simply connected Lie group $G$ by a co-compact lattice $\Gamma$, then non-flat left-invariant LCP structures on $M$ are in one-to-one correspondence with proper LCP structures on the Lie algebra $\mg$.

For the reader's convenience, we recall two basic results regarding LCP structures on Lie algebras that will be repeatedly used along this manuscript. Their proofs can be found in \cite{AdBM24}.

\begin{pro}[{\cite[Prop. 3.2]{AdBM24}}]\label{pro:algLCP} 
Let $(\mg,g)$ be a metric Lie algebra, $\theta\in\mg^*$ a non-zero closed $1$-form, and $\mmu\subset\mg$ a vector subspace.  Then $\mmu$ is $\nabla^\theta$-parallel if and only if the following two conditions hold:
\begin{enumerate}
\item $\mmu$ and $\mmu^\perp$ are Lie subalgebras;
\item for every $u\in\mmu$ and $x\in\mmu^\perp$, 
\begin{equation}\label{eq:xu}g([u,x],x)=\theta(u)|x|^2\qquad\hbox{and}\qquad g([x,u],u)=\theta(x)|u|^2.
\end{equation}
\end{enumerate}
Moreover, $(\mg,g,\theta,\mmu)$ is an LCP Lie algebra if and only if in addition to $(1)$ and $(2)$, $\mmu\neq 0$ and  the following condition holds:
\begin{enumerate}
\item[(3)] the map $\nabla^\theta:\mg\to \gl(\mmu)$, defined by $x\mapsto \nabla_x^\theta|_\mmu$, is a Lie algebra representation.
\end{enumerate}
\end{pro}

\begin{cor}[{\cite[Cor. 3.3]{AdBM24}}] \label{cor:3.4}
Let $(\mg,g,\theta,\mmu)$ be an LCP  Lie algebra and denote by $q$ and $n$ the dimensions of $\mmu$ and $\mg$ respectively. If $\mg$ is unimodular, the trace forms of $\mmu$ and $\mmu^\bot$ are related to $\theta$ by the relations
    \begin{equation}\label{eq:trace}
        H^\mmu=-(n-q) \,\theta|_\mmu,\qquad H^{\mmu^\bot}=-q\,\theta|_{\mmu^\bot}.
    \end{equation}
\end{cor}

\section{Conformally flat LCP structures}

In this section we obtain the classification of Lie algebras carrying conformally flat LCP structures, as well as the complete description of the underlying metrics and Lee forms.

It is worth noticing that every conformally flat LCP structure has vanishing Weyl tensor, and the classification of Lie algebras with vanishing Weyl tensor was given by Maier in \cite{Maier98}. However, is not possible to easily identify the conformally flat LCP Lie algebras therein. For this reason, we provide the classification of Lie algebras carrying conformally flat LCP structures in Theorem \ref{teo:krn} below, using different methods. 

For the reader's convenience, we point out that this section is independent from the following ones (except for a technical result), so it might be skipped at a first reading.

Along this section, a conformally flat LCP structure $(g,\theta,\mmu)$ on a Lie algebra $\mg$ will be denoted simply by $(g,\theta)$, since $\mmu=\mg$.
We start by providing a method to construct new conformally flat LCP structures from a given one. 

\begin{pro}\label{ex:sem} 
Let $(g_\mk,\theta_\mk)$ be a conformally flat LCP structure on a Lie algebra $\mk$ and let  $\beta:\mk\to \so(n)$ with $n\geq 0$ be any orthogonal representation.

Consider the  semidirect product $\mg=\mk\ltimes_\rho\R^n$, where  $\rho:\mk\to \gl(n)$ is the representation defined by $\rho(x):=\theta_\mk(x)\Id_n+\beta(x)$ for all $x\in \mk$.

Then the pair $(g,\theta)$, where $g=g_\mk+\langle\cdot,\cdot\rangle_{\R^n}$
and $\theta$ is the 1-form obtained by the extension of $\theta_\mk$ by zero to $\R^n$, defines a conformally flat LCP structure on $\mg$. 
\end{pro}
\begin{proof}
By construction, $\theta\neq 0$ and also $\theta|_{\mg'}=\theta_\mk|_{\mk'}=0$ because $\theta_\mk$ is closed in $\mk$. Hence $d\theta=0$. 

Condition (1) in Proposition \ref{pro:algLCP} is trivially satisfied. Condition (2) in the same proposition reduces to
\[
g([u,x],x)=\theta(u)|x|^2, \qquad u,x\in \mg.
\]
This can be easily checked by polarization and decomposing $u$ and $x$ in their $\mk$ and $\R^n$ components, using the fact that $(\mk,g_\mk,\theta_k)$ is conformally flat and the explicit expression of the representation $\rho$ defining the semidirect product.

Straightforward computations using \eqref{eq:lc} show
\[
\nabla_x^g x'=\nabla^{g_\mk}_xx', \quad \nabla_x^gy=\beta(x)y, \quad  \nabla_y^g=-\theta\wedge y, \quad  \forall x,x'\in\mk, \; y\in \R^n,\]
and thus
\[
\nabla_x^\theta x'=\nabla^{\theta_\mk}_xx', \quad \nabla_x^\theta y=\rho(x)y, \quad  \nabla_y^\theta=0, \quad  \forall x,x'\in\mk, \; y\in \R^n.
\]
Since $\nabla^{\theta_\mk}$ and $\rho$ are representations, these equalities imply that  $\nabla^\theta:\mg\to\gl(\mg)$ is a Lie algebra representation as well. Therefore,  Proposition \ref{pro:algLCP}(3) holds and hence $(\mg,g,\theta)$ is conformally flat as claimed.
\end{proof}

\begin{pro}\label{pro:iibot}
    Let $(g,\theta)$ be a conformally flat LCP structure on a Lie algebra $\mg$ and let $\mmi$ denote the kernel of the Lie algebra representation $\nabla^\theta:\mg\to\gl(\mg)$. Then $\mmi$ is an abelian ideal, $\mmi^\bot$ is Lie subalgebra of compact type and 
    \begin{equation}\label{eq:briibot}
        [x,y]=\nabla_x^\theta y , \qquad \forall\, x\in\mg, \ \forall\,y\in\mmi.
    \end{equation}     
    
     Moreover,    
    $(g|_{\mmi^\bot},\theta|_{\mmi^\bot})$ is a conformally flat LCP structure on $\mmi^\bot$.
\end{pro}
In particular, \eqref{eq:briibot} shows that $\mg$ can be written as the semidirect product $\mg=\mmi^\bot\ltimes\mmi$, where the action is given by the representation $\nabla^\theta$.
\begin{proof} By \eqref{eq:weylc1}, the kernel of $\nabla^\theta:\mg\to\gl(\mg)$ can be described as
\[
\mmi=\{x\in\mg:\nabla^g_x=x\wedge\theta, \,\theta(x)=0\}.
\]
Since $\nabla^\theta$ is a Lie algebra representation, $\mmi$ is an ideal of $\mg$. 
To show that it is abelian, consider $x,y,z\in\mmi$; the conditions $\nabla^g_x=x\wedge\theta$ and $\theta|_\mmi=0$ imply 
\[
\lela\nabla^g_xy,z\rira=-\lela (\theta\wedge x)y,z\rira=-\lela \theta(y)x-\lela x,y\rira\theta^\sharp,z\rira=0.
\]Hence $\lela [x,y],z\rira=\lela\nabla_x^gy-\nabla_y^gx,z\rira=0$ for all  $x,y,z\in\mmi$ and thus $\mmi$ is abelian.

Now, taking $x\in\mmi$, $y,z\in\mmi^\bot$ in \eqref{eq:LCP}, and using the fact that $\mmi$ is an ideal, $\nabla_x^\theta=0$, and $\theta|_{\mmi}=0$, we get
   \[
0=\lela\nabla^\theta_xy,z\rira=-\frac12\lela[y,z],x).
\]Therefore, $\mmi^\bot$ is a Lie subalgebra of $\mg$. The representation $\nabla^\theta$ restricts to an injective representation of $\mmi^\bot$ into a Lie algebra of compact type. Thus $\mmi^\bot$ is of compact type, and in particular unimodular.

Next we prove \eqref{eq:briibot}. If $x,y\in\mmi$, then $[x,y]=0$ and $\nabla_x^\theta y=0$ by the definition of $\mmi$, so the equality holds in this case. Now if $x\in\mmi^\bot$ and $y\in\mmi$, $[x,y]\in\mmi$ and for any $z\in\mmi$, \eqref{eq:weylc1} gives
\begin{eqnarray*}
\lela \nabla_x^\theta y,z\rira&=&\lela \nabla^g_xy+\theta(x)y,z\rira
=\lela \nabla^g_yx+[x,y]+\theta(x)y,z\rira\\
&=&\lela (y\wedge \theta)x+[x,y]+\theta(x)y,z\rira=\lela[x,y],z\rira,
\end{eqnarray*}showing the equality in all cases.

It remains to show that  $(g|_{\mmi^\bot},\theta|_{\mmi^\bot})$ is conformally flat. For this, let $x,y\in \mmi^\bot$ and $z\in\mmi$. Then $\theta|_\mmi=0$ together with \eqref{eq:weylc1} give
\[
\lela \nabla_x^\theta y,z\rira=\lela \nabla^g_xy+(\theta\wedge x)y+\theta(x)y,z\rira=\lela \nabla^g_xy,z\rira,
\]and this is zero because of Koszul's formula and the facts that $\mmi$ is an ideal and $\mmi^\bot$ a subalgebra. Therefore, $\nabla^g_xy\in\mmi^\bot$ for all $x,y\in\mmi^\bot$ and thus $\nabla^{\theta}$ defines a Lie algebra representation from $\mmi^\bot$ to $\gl(\mmi^\bot)$. Hence  $(g|_{\mmi^\bot},\theta|_{\mmi^\bot})$ is a conformally flat LCP structure by Proposition \ref{pro:algLCP}.
\end{proof}

\begin{cor}\label{cor:equiv}
   Let $(g,\theta)$ be a conformally flat LCP structure on a Lie algebra $\mg$. The following statements are equivalent:
   \begin{enumerate}
       \item $\mg$ is unimodular;
       \item $\mg$ is of compact type;
       \item $\nabla^\theta:\mg\to\gl(\mg)$ is injective.
   \end{enumerate}
\end{cor}
\begin{proof}As before, let $\mmi$ denote the kernel of $\nabla^\theta:\mg\to\gl(\mg)$. By Proposition \ref{pro:iibot},  $\mmi^\bot$ is of compact type, thus unimodular so $\tr(\ad_x|_{\mmi^\bot})=0$ for all $x\in\mmi^\bot$. 
Moreover, the adjoint action of $\mmi^\bot$ preserves both $\mmi$ and $\mmi^\bot$, therefore for any $x\in\mmi^\bot$,
\begin{equation}\label{eq:tri}
\tr\ad_x=\tr(\ad_x|_{\mmi^\bot})+\tr(\ad_x|_{\mmi})=\tr(\ad_x|_{\mmi}).
\end{equation} In addition, equation \eqref{eq:briibot} together with \eqref{eq:weylc1} imply that for any $x\in\mmi^\bot$, $y\in \mmi$, 
\[
[x,y]=\theta(x)y+\beta(x)y, \quad \mbox{ for some }\beta(x)\in\so(\mmi).
\]This together with \eqref{eq:tri} implies
\[
\tr\ad_x=\theta(x)\dim\mmi, \quad \forall x\in\mmi^\bot,
\]where actually, $\theta|_{\mmi^\bot}\neq 0$ because $\theta|_\mmi=0$ and $\theta\neq 0$. Therefore, $\mg$ is unimodular if and only if $\mmi=0$. This is also equivalent to $\mg=\mmi^\bot$ which is of compact type.
\end{proof}

In the next examples, we construct conformally flat LCP structures on some unimodular Lie algebras.

\begin{ex}
    \label{ex:rp} Abelian Lie algebras of dimension 1 or 2 admit conformally flat LCP structures. 

    Indeed, given $p\in\{1,2\}$, an inner product $g$ on $\R^p$ and $\theta\in(\R^p)^*$, the Weyl connection defined by $\theta$ satisfies $\nabla_x^\theta=\theta\wedge x+\theta(x)\Id_{\R^p}$ for every $x\in\R^p$. For dimensional reasons, $[\theta\wedge x,\theta\wedge y]=0$ for every $x,y\in\R^p$, and thus $\nabla^\theta:\R^p\to \gl(p)$ is a Lie algebra representation. This shows that Proposition \ref{pro:algLCP}(3) holds, whilst (1) and (2) in that proposition hold trivially. Therefore,  $(g,\theta)$ is a conformally flat structure on $\R^p$.
\end{ex}

\begin{ex}
    \label{ex:su2r}    
Consider the direct sum Lie algebra $\mg:=\su(2)\oplus\R$, and fix an element $z$ spanning its center. Let $\kappa$ denote the Killing form of $\su(2)$, which is negative definite. 

Choose $\mu>0$, $0\neq\lambda\in\R$ and $x_0\in\su(2)$; we shall define a conformally flat LCP structure on $\mg$ using these parameters.

Let $g_{\mu,\lambda}$ be the metric on $\mg$ such that $g_{\mu,\lambda}|_{\su(2)}=-\mu\kappa $ and
\begin{eqnarray}
    g_{\mu,\lambda}(z,x)&=&\frac{\mu}{\lambda}\kappa(x_0,x)=-\frac1{\lambda}g_{\mu,\lambda}(x_0,x), \qquad\forall x\in\su(2),\label{eq:gml}\\
    |z|^2_{g_{\mu,\lambda}}&=&\frac1{\lambda^2}(\frac1{8\mu}-\mu\kappa(x_0,x_0))=\frac1{\lambda^2}(\frac1{8\mu}+|x_0|_{g_{\mu,\lambda}}^2).\label{eq:znorm}
\end{eqnarray}
From \eqref{eq:gml}, we get $g_{\lambda,\mu}(\lambda z+x_0,x)=0$ for all $x\in\su(2)$, which implies that the 1-form $\theta:=g_{\mu,\lambda}(\lambda z+x_0, \cdot)$ is closed in $\mg$. 
Notice that $\theta(z)=1/(8\mu\lambda)$ and $|\theta|^2=1/(8\mu)$; the latter value coincides with the the sectional curvature of the metric $-\mu\kappa$ in $\su(2)$. 

Using Koszul's formula, one gets for every $x\in\su(2)$:
\begin{equation}\label{eq:nabla}
\nabla_z^{g_{\mu,\lambda}}z=0,\quad \nabla_x^{g_{\lambda,\mu}}=\frac12\ad_x\in\so(\su(2)),\quad \nabla_x^{g_{\lambda,\mu}}z= \nabla_z^{g_{\lambda,\mu}}x=\frac{1}{2\lambda}[x_0,x].
\end{equation}
This implies, in particular, that $\nabla_w^{g_{\mu,\lambda}}\theta=0$ for all $w\in\mg$ and therefore, by \eqref{eq:weylc1},
\begin{equation}\label{eq:conm}
[\nabla_u^\theta,\nabla_w^\theta]=[\nabla_u^{g_{\mu,\lambda}},\nabla_w^{g_{\mu,\lambda}}]+\theta\wedge [u,w]+|\theta|^2u\wedge w+\theta\wedge (\theta(w)u-\theta(u)w)
\end{equation}
for any $u,w\in\mg$. Here and henceforth we denote by $u\wedge w$ the skew-symmetric endomorphism of $\mg$ defined as $(u\wedge w)(x)=g(u,x)w-g(w,x)v$, for all $x\in\mg$.

For any $x\in\su(2)$, \eqref{eq:conm}, together with the facts that $\theta(x)=0$, $z$ is central and $\theta=\lambda z+x_0$,  give
\begin{equation}\label{eq:cxz}
        [\nabla_x^\theta,\nabla_z^\theta]=[\nabla_x^{g_{\mu,\lambda}},\nabla_z^{g_{\mu,\lambda}}]+\frac{1}{8\mu\lambda}x_0\wedge x.
\end{equation}
On the one hand, straightforward computations using \eqref{eq:nabla} show that $[\nabla_x^{g_{\mu,\lambda}},\nabla_z^{g_{\mu,\lambda}}] (w)=\frac1{4\lambda}\ad_{[x,x_0]}w$, for every $w\in\mg$. On the other hand, making use the metric $g_{\mu,\lambda}$, we can identify the skew-symmetric endomorphisms
\begin{equation}\label{eq:ident}
   \ad_{[x,x_0]} =\frac1{2\mu} x\wedge x_0.
\end{equation} Therefore, \eqref{eq:cxz} becomes
\[
    [\nabla_x^\theta,\nabla_z^\theta]=\frac1{4\lambda}\ad_{[x,x_0]}+\frac{1}{8\mu\lambda}x_0\wedge x=\frac1{8\mu\lambda}x\wedge x_0+\frac{1}{8\mu\lambda}x_0\wedge x=0=\nabla^\theta_{[x,z]}.
\]

Now, given $x,y\in\su(2)$, since $\theta|_{\su(2)}=0$ and \eqref{eq:nabla} holds, \eqref{eq:conm} and \eqref{eq:ident} imply
\begin{equation*}\label{eq:cxy}
[\nabla_x^\theta,\nabla_y^\theta]=\frac14\ad_{[x,y]}+\theta\wedge [x,y]+\frac{1}{8\mu}x\wedge y=\frac12\ad_{[x,y]}+\theta\wedge [x,y]=\nabla^\theta_{[x,y]}.
\end{equation*}
These last two equations show that $\nabla^\theta:\mg\to\gl(\mg)$  is a representation so $(g_{\mu,\lambda},\theta)$ is a conformally flat LCP structure on $\su(2)\oplus \R$ by Proposition \ref{pro:algLCP}.
\end{ex}

We will show that the above examples exhaust all possible conformally flat LCP structures on unimodular Lie algebras. To do so, we first introduce a technical result.

\begin{lm}\label{lm:cflat}
    Let $(g,\theta)$ be a conformally flat LCP structure on a Lie algebra $\mg$ and let $R^g$ be the curvature tensor of the Levi-Civita connection $\nabla^g$. Then, the following equations hold:
    \begin{eqnarray}
          & \label{eq:2.1} R_{x,y}^g=-|\theta|^2x\wedge y+\theta(x)\theta\wedge y -\theta(y) \theta\wedge x - \nabla_x^g\theta \wedge y+\nabla_y^g\theta\wedge x,\quad \forall x,y\in\mg,\\
            & \label{eq:2.2} |\nabla^g \theta|^2=\lela H^\mg,\theta\rira.
    \end{eqnarray}
\end{lm}
\begin{proof} Assume $(g,\theta)$ is a conformally flat structure on $\mg$ and consider the representation $\nabla^\theta:\mg\to\gl(\mg)$ given in Proposition \ref{pro:algLCP}. By \eqref{eq:weylc1}, we have
\begin{equation*}
    \nabla_x^\theta=\theta(x)\Id_\mg+\beta(x),\qquad \forall x\in\mg,
\end{equation*}where $\beta:\mg\to\so(\mg)$ is defined by
\begin{equation*}
    \beta (x):=\nabla_x^g+\theta\wedge x\in\so(\mg).
\end{equation*} It is easy to check that $\beta$ is a representation of $\mg$ because $\nabla^\theta$ is so and $\theta$ is closed. Therefore, for every $x,y\in\mg$,
\begin{multline*}
    0=[\beta(x),\beta(y)]-\beta([x,y])=[\nabla_x^g,\nabla_y^g]+[\nabla_x^g,\theta\wedge y]+[\theta\wedge x,\nabla_y^g]+[\theta\wedge x,\theta\wedge y]-\nabla_{[x,y]}^g-\theta\wedge [x,y]\\
    =R_{x,y}^g+\nabla_x^g\theta \wedge y+\theta\wedge \nabla_x^gy-\nabla_y^g\theta\wedge x-\theta\wedge\nabla_y^gx+|\theta|^2x\wedge y-\theta(x)\theta\wedge y+\theta(y) \theta \wedge x-\theta\wedge 
[x,y]\\
=R_{x,y}^g+\nabla_x^g\theta \wedge y-\nabla_y^g\theta\wedge x+|\theta|^2x\wedge y-\theta(x)\theta\wedge y+\theta(y) \theta \wedge x
\end{multline*}
which implies
\[
R_{x,y}^g=-|\theta|^2x\wedge y+\theta(x)\theta\wedge y -\theta(y) \theta\wedge x - \nabla_x^g\theta \wedge y+\nabla_y^g\theta\wedge x, \quad \forall x,y\in\mg,
\]giving the first equality in the statement.

With this expression for the curvature, we can thus compute the Ricci tensor of $g$, obtaining
\begin{equation}\label{eq:ric}
\Ric^gx =|\theta|^2(n-2)x-(n-2)\theta(x)\theta^\sharp+(n-2)\nabla^g_x\theta^\sharp-x\delta^g \theta,\qquad\forall x\in\mg,
\end{equation} where $\delta^g$ denotes the codifferential of $g$.

Since $\theta$ is closed and left-invariant, $\Delta \theta=0$, thus \eqref{eq:ric} applied to $x:=\theta^\sharp$ together with the Bochner formula for $\theta$ yields
\begin{equation}\label{eq:boch}
0=(\nabla^g)^*\nabla^g\theta^\sharp+\Ric^g\theta^\sharp = (\nabla^g)^*\nabla^g\theta^\sharp+(n-2)\nabla^g_{\theta^\sharp}\theta^\sharp-\theta^\sharp \delta^g \theta.
\end{equation}
We shall compute the last two terms of this expression.
On the one hand, for every $x\in\mg$,
\[
\lela\nabla^g_{\theta^\sharp}\theta^\sharp,x\rira=\lela [x,\theta^\sharp],\theta^\sharp\rira=\theta([x,\theta^\sharp])=0
\]because $d\theta=0$. On the other hand, if $\{e_i\}_{i=1}^n$ is an orthonormal basis of $\mg$, \[
\delta^g\theta=-\sum_{i=1}^n\lela e_i,\nabla^g_{e_i}\theta^\sharp\rira=\tr \ad_{\theta^\sharp}=\lela H^\mg,\theta).
\]
Therefore, \eqref{eq:boch} becomes
\[
 (\nabla^g)^*\nabla^g\theta=\lela H^\mg,\theta)\theta.
\]Taking the inner product with $\theta$ we obtain
\[
\lela H^\mg,\theta)|\theta|^2=\lela (\nabla^g)^*\nabla^g\theta,\theta\rira=\sum_{i=1}^n\lela -\nabla_{e_i}^g\nabla_{e_i}^g\theta+\nabla^g_{\nabla^g_{e_i}e_i}\theta,\theta\rira=\sum_{i=1}^n|\nabla_{e_i}^g\theta|^2=|\nabla^g\theta|^2,
\]because $\nabla^g_{x}$ is skew-symmetric for any $x\in\mg$. This shows the second formula.
\end{proof}


\begin{pro}\label{pro:cflatunim}
    Let $(g,\theta)$ be a conformally flat LCP structure on a unimodular Lie algebra $\mg$. Then $\mg$ is isomorphic to one of the Lie algebras:
    \[
\R, \qquad \R^2,\qquad \su(2)\oplus\R,  
    \]and the LCP structure is as in Examples \ref{ex:rp} or \ref{ex:su2r}, respectively.
    In particular, in the last case, the metric $g|_{\su(2)}$ is a multiple of the Killing form.
\end{pro}
\begin{proof}
Assume that $(\mg,g,\theta)$ is a conformally flat LCP Lie algebra and $\mg$ is unimodular, that is, $H^\mg=0$. By \eqref{eq:2.2} we obtain $\nabla^g \theta=0$, so \eqref{eq:2.1} becomes
\begin{equation}\label{eq:Runim}
 R_{x,y}^g=-|\theta|^2x\wedge y+\theta(x) \theta \wedge y-\theta(y) \theta \wedge  x,\quad \forall x,y\in\mg.    
\end{equation}

Consider the codimension 1 subspace $\mk:=\ker \theta$ of $\mg$ (which is an ideal due to $d\theta=0$). This defines by left translations a totally geodesic distribution. Indeed,
\[
\theta( \nabla^g_xy)=-(\nabla^g_x\theta)(y)=0, \qquad \forall x,y\in\mk,
\]
because $\nabla^g\theta=0$. Therefore, by \eqref{eq:Runim}, the curvature of $(\mk,g|_\mk)$
is $R_{x,y}^g=-|\theta|^2x\wedge y$, so the corresponding simply connected Lie subgroup $K$ with Lie algebra $\mk$ has a metric of constant positive sectional curvature $|\theta|^2$. 

This implies that $K$ is diffeomorphic to the sphere $\mathbb S^{n-1}$, and thus $\mathbb S^{n-1}$ has a Lie group structure. Hence $n-1$ is either 0, 1, or 3, and $\mk=0$, $\mk=\R$ or $\mk=\su(2)$, respectively.

As $\mg$ is of compact type (by Corollary \ref{cor:equiv}), if $n=\dim\mg=1,2$, the only possibilities are $\mg=\R$ or $\mg=\R^2$. It is clear that the conformally flat structure is as in Example \ref{ex:rp}.

If $\dim\mg=4$, then $\mk=\su(2)$ and thus $\mg=\su(2)\oplus\R$, with $g|_{\su(2)}$ being of constant sectional curvature. Lemma \ref{lm:bii} below shows that $g|_{\su(2)}$ is a multiple of the Killing form, fact that we assume for the moment. In this case, $g|_{\su(2)}=-\mu\kappa$ with $\mu:=\frac1{8|\theta|^2}$, since the sectional curvature of $g|_{\su(2)}$ is $|\theta|^2$.

Fix $z$ an element spanning the center of $\mg=\su(2)\oplus\R$, and write $\theta^\sharp=\lambda z+x_0$ with $\lambda\neq 0$ and $x_0\in\su(2)$. One can readily check that the metric $g$ is $g_{\lambda,\mu}$ and thus the conformally flat structure is as in Example \ref{ex:su2r}.
\end{proof}

\begin{remark}\label{rem:latcf} Notice that all simply connected Lie groups corresponding to unimodular Lie algebras admitting conformally flat LCP structures carry co-compact lattices.
\end{remark}
We now prove the result claimed in the proof above.

\begin{lm} \label{lm:bii}
    A left-invariant metric on $\mathrm{SU}(2)$ has constant sectional curvature if and only if it is bi-invariant.
\end{lm}
\begin{proof} The converse statement being obvious, assume that that $g$ is a scalar product on $\su(2)$ such that the induced left-invariant metric $g$ on $\mathrm{SU}(2)$ has constant sectional curvature. Up to rescaling, we can assume that $\mathrm{SU}(2)$ is isometric to the round sphere $\mathbb{S}^3$. 

The right-invariant vector fields on $\mathrm{SU}(2)$ are Killing vector fields. This gives a Lie algebra (anti-) morphism $f:\su(2)\to\mathfrak{isom}(\mathbb{S}^3)\simeq \su(2)\oplus\su(2)$ such that for every non-zero $x\in \su(2)$, the Killing vector field on $\mathbb{S}^3$ determined by $f(x)$ has no zeros. 

Recall that upon identification of $\su(2)$ with the Lie algebra of imaginary quaternions, and of $\mathbb{S}^3$ with the sphere of unit quaternions, the Killing vector field on $\mathbb{S}^3$ determined by an element $(a,b)\in\su(2)\oplus\su(2)$ is given by 
\begin{equation}\label{eq:xi}
    \xi(x,y)_v:=x\cdot v - v\cdot y,\qquad\forall\, v\in \mathbb{S}^3\subset\H,
\end{equation} where $\H$ denotes the algebra of quaternions and $\cdot$  denotes their product.

Now, let us write the Lie algebra morphism $f$ as $f=(f_1,f_2)$, for Lie algebra morphisms $f_i:\su(2)\to\su(2)$. We claim that one of $f_1$ and $f_2$ vanishes. 

Assume for a contradiction that both $f_1$ and $f_2$ are non-zero. Since $\su(2)$ is simple, $f_i$ are both automorphisms of $\su(2)$, so there exist $a,b\in\mathrm{SU}(2)\simeq \mathbb{S}^3\subset\H$ such that $f(x)=(a\cdot x\cdot a^{-1},b\cdot x\cdot b^{-1})$ for every $x\in\su(2)$. Then for every imaginary quaternion of unit length $x\in \su(2)\cap\mathbb{S}^3$, the vector field $\xi(f(x))$ of $\mathbb{S}^3$ vanishes at $v:=a\cdot x\cdot b^{-1}\in \mathbb{S}^3$ by \eqref{eq:xi}. This is a contradiction, thus proving our claim.

It follows that there exists a 3-dimensional Lie algebra $\mh\subset\mathfrak{isom}(\mathrm{SU}(2),g)$ all of which elements commute with the image of $f$, i.e. with the right-invariant vector fields. On the other hand, every vector field on a connected Lie group commuting with the right-invariant vector fields is left-invariant. For dimensional reasons we get that $\mh$ is equal to the Lie algebra of left-invariant vector fields on $\su(2)$, which are therefore Killing. This shows that the metric $g$ is bi-invariant.
\end{proof}

Notice that Lemma \ref{lm:bii} also follows from \cite[Lemma 2.7]{PKP09}, by an argument using brute force computations of the Ricci curvature of a left-invariant metric on $\mathrm{SU}(2)$ in an adapted orthogonal frame. 

The results above account to the following classification of conformally flat LCP structures on Lie algebras. 

\begin{teo} \label{teo:krn} Let $\mg$ be a Lie algebra admitting a conformally flat LCP structure $(g,\theta)$. Then $\mg$ is isomorphic to a Lie algebra
\[
\mk\ltimes \R^n, 
\]for some $n\geq 0$, with $\mk$ either $\R^p$ for $p\in\{1,2\}$ or $\su(2)\oplus \R$. Moreover, the conformally flat structure in $\mg$ is an extension as in Proposition \ref{ex:sem} of the conformally flat structures in Examples \ref{ex:rp} and \ref{ex:su2r}.
\end{teo}
\begin{proof} Let $(g,\theta)$ be a conformally flat LCP structure on $\mg$. By Proposition \ref{pro:iibot}, $\mg=\mk\ltimes \mmi$, with $\mmi:=\ker \nabla^\theta$ an abelian ideal  and $\mk:=\mmi^\bot$ of compact type and conformally flat.
Recall that $\theta|_\mmi=0$, so \eqref{eq:briibot} and \eqref{eq:weylc2} imply 
\[
\nabla_x^\theta|_{\mmi}=\ad_x|_\mmi=\beta(x)+\theta(x)\Id_\mmi,\qquad \forall x\in\mmi^\bot,
\] where $\beta(x):=\nabla_x^g|_\mmi\in\so(\mmi)$. Since $\nabla^\theta$ is a representation, the map $\rho:\mk\to\gl(\mmi)$,  $\rho(x):=\beta(x)+\theta(x)\Id_\mmi$ is representation as well, and defines the action in the semidirect product $\mk\ltimes\mmi$.

In addition, $\theta$ vanishes on $\mmi$, so $\theta\in\mk^*\subset ( \mk\ltimes \mmi)^*$. Using Proposition \ref{pro:algLCP}, one can easily check that this 1-form $\theta\in\mk^*$ together with the metric $g|_{\mk}$ define a conformally flat structure on $\mk$. Therefore, by Proposition \ref{pro:cflatunim}, the structure in $(\mk,g|_{\mk})$ is the one given  either in Example \ref{ex:rp} or \ref{ex:su2r}, and thus the one in $(\mg,g)$ is an extension of those by $\rho$ above, as described in Proposition \ref{ex:sem}.
\end{proof}

\section{Proper LCP structures}\label{sec:proper}

Having settled the conformally flat case in the previous section, we now focus on proper LCP structures. Let us first introduce the following important subclass of LCP structures.

\begin{defi}
An LCP structure $(g, \theta, \mmu)$ on a Lie algebra $\mg$ is called {\em adapted} if $\theta|_\mmu=0$.   
\end{defi}
Note that any adapted LCP structure is proper, since $\theta \neq 0$. 

Most of the results in this section only hold for adapted LCP structures. However, in view of Theorem \ref{teo:unim} below, this is not a restrictive assumption. Indeed, we are primarily interested in compact quotients of Riemannian Lie groups carrying LCP structures, and by \cite{Mil76}, the Lie algebra of any Lie group admitting co-compact lattices is unimodular.

\begin{teo}\label{teo:unim}
Every proper LCP structure $(g, \theta, \mmu)$ on a unimodular Lie algebra $\mg$ is adapted.
\end{teo}
\begin{proof}
Let $(g,\theta,\mmu)$ be a proper LCP structure on $\mg$. Then $\mmu$ is a subalgebra of $\mg$ by Proposition \ref{pro:algLCP}(1) and  $\theta|_\mmu$ is a closed 1-form in $\mmu$. Assume for a contradiction that it is not adapted, i.e. $\theta|_\mmu\neq 0$. We claim that in this case $(g|_{\mmu},\theta|_{\mmu})$ is a conformally flat LCP structure on $\mmu$.

To show this, it suffices to show that Proposition \ref{pro:algLCP} holds for $(\mmu,g|_{\mmu},\theta|_{\mmu},\mmu)$.
Notice that for any $x,y,v\in\mmu$, using \eqref{eq:LCP} for $(g,\theta,\mmu)$ together with the facts that $\mmu$ is a subalgebra, and the metric and the 1-form in $(\mmu,g|_{\mmu},\theta|_{\mmu},\mmu)$ are just restrictions of the elements in $\mg$, we obtain
    \begin{multline*}
    \lela\nabla^\theta_xy,z\rira=\frac12\left(\lela  [x,y],z\rira-\lela[x,z],y\rira-\lela[y,z],x)\right)\\
+\theta(x)\lela y,z\rira+\theta(y)\lela x,z\rira-\theta(z)g(x,y)\\
=\frac12\left(g|_{\mmu}( [x,y],z)-g|_{\mmu}([x,z],y)-g|_{\mmu}([y,z],x)\right)\\
+\theta|_{\mmu}(x)\lela y,z\rira+\theta|_{\mmu}(y)\lela x,z\rira-\theta|_{\mmu}(z)g(x,y)=\lela\nabla^{\theta|_{\mmu}}_xy,z\rira
\end{multline*}
Recalling that $\mmu$ is $\nabla^\theta$-parallel, we get from \eqref{eq:LCP}
\[
\nabla^\theta_xy=\nabla^{\theta|_{\mmu}}_xy,\qquad \forall x,y\in\mmu.
\] This implies that $\nabla^{\theta|_{\mmu}}:\mmu\to \gl(\mmu)$ is nothing but the restriction to $\mmu$ of the representation 
$\nabla^\theta:\mg\to \gl(\mmu)$, so Condition (3) in Proposition \ref{pro:algLCP} holds for $(\mmu,g|_{\mmu},\theta|_{\mmu},\mmu)$. Moreover, Condition (1) and the left-hand-side of Condition (2) in the same proposition are trivially satisfied, whilst the right-hand-side of the latter follows directly from the fact that $(g, \theta,\mmu)$ is an LCP structure on $\mg$, thus showing that $(g|_{\mmu},\theta|_{\mmu})$ is a conformally flat LCP structure on $\mmu$.

By \eqref{eq:trace}, we have $H^\mmu=-(n-q)\theta|_{\mmu}$, where $n=\dim \mg$ and $q=\dim \mmu$. In addition, \eqref{eq:2.2} applied to $(\mmu,g|_{\mmu},\theta|_{\mmu},\mmu)$ gives
  \[|\nabla^{g|_\mmu} \theta|_{\mmu}|^2=\lela H^\mmu,\theta|_\mmu\rira=-(n-q)|\theta|_\mmu|^2,\] 
  which is a contradiction since we assumed $\theta|_\mmu\neq 0$ and $n-q> 0$ because the LCP structure is proper. This shows that $\theta|_\mmu=0$, thus finishing the proof.
\end{proof}

\begin{remark}The previous theorem generalizes \cite[Lemma 5.2]{AdBM24}, where it was shown that any LCP structure on a solvable unimodular Lie algebra is adapted.
\end{remark}

The next example shows that the unimodularity assumption in Theorem \ref{teo:unim} is necessary. 

\begin{ex}Let $\mg=\R b\ltimes\R^n$ be the semidirect product of abelian Lie algebra where the action of $b$ on $\R^n$ is $\ad_b:=\Id_{\R^n}$, and consider the inner product $g$ on $\mg$ obtained by extending the standard metric on $\R^n$ orthogonally, and such that $b$ is of unit norm. Then, taking $\theta:=g(b,\cdot)$ and using \eqref{eq:LCP}, we get that $\nabla_x^\theta=0$ for all $x\in\R^n$ and thus $(g,\theta,\mmu:=\R b)$ is a non-adapted LCP structure on $\mg$ which is not conformally flat either.
Notice that the Lie algebra $\mg$ is not unimodular.
\end{ex}

The following result generalizes Proposition 6.1 in \cite{AdBM24}. For the reader's convenience we include its proof here, even though it is similar to the original one.
\begin{cor}\label{cor:bound}
Every proper LCP structure $(g,\theta,\mmu)$ on an $n$-dimensional unimodular Lie algebra $\mg$ satisfies $\dim \mmu\leq n-2$.
\end{cor}
\begin{proof} 
    The LCP structure is proper, so $1\le\dim \mmu\leq n-1$, whence $n\geq 2$. Moreover,  \eqref{eq:xu} implies that 
    \begin{equation}\label{eq:tia}
    \ad_{\theta^\sharp}|_{\mmu}=|\theta|^2\Id_\mmu+A 
    \end{equation}
    where $A\in \so(\mmu)$ in view of \eqref{eq:weylc2}.

Assume for a contradiction that $\dim \mmu=n-1$. Then $\mg=\R\theta^\sharp\oplus \mmu$ and therefore
\[
0=\tr(\ad_{\theta^\sharp})=\tr(\ad_{\theta^\sharp}|_\mmu).
\]However, using \eqref{eq:tia} in the last equation we get
\[
0=(n-1)|\theta|^2,
\] contradicting the fact that  $n\geq 2$ and $\theta\neq 0$.
 \end{proof}

\begin{remark}In contrast to the conformally flat case (see Remark \ref{rem:latcf}), there exist unimodular Lie algebras carrying proper LCP structures, that do not admit co-compact lattices. Indeed, this was already shown in \cite{AdBM24} for solvable unimodular Lie algebras. In general, the problem of finding such lattices on unimodular LCP Lie algebras is highly non-trivial. Examples of non-solvable Lie groups carrying co-compact lattices and whose Lie algebras admit LCP structures will be constructed in Section \ref{sec:nonsolv}.
\end{remark}

The rest of the section is devoted to the study of adapted LCP Lie algebras, not necessarily unimodular (even though by Theorem \ref{teo:unim}, every LCP structure on a unimodular Lie algebra is adapted).
We start with a necessary and sufficient condition for a proper LCP structure to be adapted.

\begin{pro}\label{pro:adid} A proper LCP structure $(g,\theta,\mmu)$ on a Lie algebra $\mg$ is adapted if and only if $\mmu$ is an ideal. If this holds, then $\mmu$ is abelian and contained in $\mg'$, thus in particular $\mmu$ is contained in the nilradical of $\mg$.
\end{pro}
\begin{proof} Let $(g,\theta,\mmu)$ be a proper LCP structure on $\mg$. We shall first prove the equivalence in the statement. 
    Assume that  $\mmu\neq \mg$ is an ideal and let $x\in\mmu^\bot$ be of unit norm. Then, by \eqref{eq:xu} we get
    \[
0=g([u,x],x)=\theta(u), \qquad \forall \,u\in\mmu, 
    \]showing that the LCP structure is adapted.

For the converse, assume that $(g,\theta,\mmu)$ verifies $\theta|_{\mmu}=0$. Notice that if $(g,\theta,\mmu)$ is an LCP structure on $\mg$, then for any $\lambda>0$, $(\lambda g,\theta,\mmu)$ is LCP as well, and the first one is adapted if and only if the latter is so. Hence, by rescaling the metric $g$, we assume that $|\theta|=1$ without loss of generality.

Consider the orthogonal decomposition $\mg=\mmu\oplus\R\ts\oplus \mmp$, where $\mmu\oplus \mmp$ is an ideal of $\mg$ because it contains $\mg'$ (due to $d \theta=0$). This, together with the fact that $\mmu^\bot=\R\ts\oplus\mmp$ is a subalgebra, by Proposition \ref{pro:algLCP}, implies 
\[[\ts,\mmp]\subset \mmp,\]
 and thus there exists $A\in\gl(\mmp)$ such that
 \begin{equation}
\label{eq:thepp}
[\ts,x]=Ax, \qquad\forall x\in \mmp.
\end{equation}
Moreover, \eqref{eq:xu} gives for any $u\in\mmu$
\[0=g([u,\ts],\ts), \quad \mbox{ and } \quad 
0=g([u,x+\ts],x+\ts)=g([u,\ts],x), \quad \forall\, x\in\mmp,
\] so that $[\ts,\mmu]\subset \mmu$. Also from  \eqref{eq:xu} we get
\begin{equation*}
       g([\ts,u],u)=|u|^2,
\end{equation*} since $|\theta|=1$.
The last two equations together imply
\begin{equation}
     \label{eq:thu}
     [\ts,u]=u+Bu, \mbox{ for some } B\in\so(\mmu).
\end{equation}

Moreover, since $\theta(\mmu\oplus\mmp)=0$, \eqref{eq:xu} gives
\[
g([u,x],x)=0=g([x,u],u), \qquad \forall\, u\in\mmu,\, x\in\mmp,
\] which implies, for every $u\in\mmu$, $x\in \mmp$,
\begin{equation}
     \label{eq:pu}
     [x,u]=B_xu+C_ux, \mbox{ for some } B_x\in\so(\mmu), \,C_u\in\so(\mmp).
\end{equation}

Using \eqref{eq:thepp}, \eqref{eq:thu} and \eqref{eq:pu} in the Jacobi identity for $\ts$, $u\in\mmu$, $x\in \mmp$,
 we get
 \begin{eqnarray*}
     [[\ts,x],u]&=&[\ts,[x,u]]-[x,[\ts,u]]\\
     B_{Ax}u+C_uAx&=& (\Id_\mmu+B)B_xu+AC_ux-B_x(u+Bu)-C_{Bu+u}x\\
     (C_uA-AC_u+C_{Bu+u})x&=& (BB_x-B_xB-B_{Ax})u.
 \end{eqnarray*}
 where the left-hand side is in $\mmp$ whilst the right hand side is in $\mmu$, so both must vanish. In particular, this implies
 \begin{equation}\label{eq:ACu}
     [A,C_u]=C_u+C_{Bu}, \quad \mbox{ for all }u\in \mmu.
 \end{equation}

Let $q\ge 1$ denote the dimension of $\mmu$ and let $\{u_i\}_{i=1}^q$ be an orthonormal basis of $\mmu$. 
Since $\tr(C_u[A,C_u])=0$ for all $u\in\mmu$, by \eqref{eq:ACu} we get
\[
0=\sum_{i=1}^q\tr(C_{u_i}[A,C_{u_i}])=\sum_{i=1}^q\tr(C_{u_i}^2)+\sum_{i=1}^q\tr (C_{u_i}C_{B_{u_i}}).
\]However, the last term vanishes. Indeed, since $B$ is skew-symmetric, 
\[
\sum_{i=1}^q \tr(C_{u_i}C_{Bu_i})=\sum_{i,j=1}^q\lela Bu_i,u_j\rira \tr(C_{u_i}C_{u_j})=0.
\]
We thus conclude $\tr(C_{u_i}^2)=0$ for all $i=1, \ldots, q$, but  $C_{u_i}$ being skew-symmetric, we actually get that they all must vanish. Hence, by \eqref{eq:pu}, $\mmp$ preserves $\mmu$, that is, $[\mmp,\mmu]\subset \mmu$. Using the fact that $\mmu$ is a subalgebra of $\mg$, together with $[\ts,\mmu]\subset \mmu$ which follows from \eqref{eq:thu},  shows that $\mmu$ is an ideal of $\mg$ as claimed.

For the last part, using \eqref{eq:thu} and the Jacobi identity applied to $u,v\in\mmu$ and $\ts$, we get for every $u,v\in\mmu$
\begin{eqnarray*}
 \,   [\ts,[u,v]]&=&[[\ts,u],v]+[u,[\ts,v]],\\
 \, \mathrm{whence}\quad B[u,v]&=&[Bu,v]+[u,Bv]+[u,v].
\end{eqnarray*}Since $B$ is skew-symmetric, this equation gives
\begin{eqnarray*}
       0&=&\sum_{i,j=1}^q\lela B[u_i,u_j],[u_i,u_j]\rira \\
    &=&\sum_{i,j=1}^q\lela [Bu_i,u_j],[u_i,u_j]\rira+\sum_{i,j=1}^q\lela [u_i,Bu_j],[u_i,u_j]\rira+\sum_{i,j=1}^q\lela [u_i,u_j],[u_i,u_j]\rira\\
    &=&\sum_{i,j=1}^q|[u_i,u_j]|^2,
\end{eqnarray*}where we used a change of order in the summation indexes in the second term. Therefore, $\mmu$ is abelian. The fact that $\mmu\subset\mg'$ follows directly from \eqref{eq:thu}, since $\Id_\mmu+B$ is invertible for every $B\in\so(\mmu)$.
\end{proof}

\begin{cor}
Let $(g , \theta,\mmu)$ be  an adapted LCP structure on a Lie algebra $\mg$. For any ideal $\mr\subset\mg$ such that $\mmu\subset\mr$ and $\theta|_{\mr}\neq0$, the triple $(g|_{\mr},\theta|_{\mr},\mmu)$  is an adapted LCP structure on $\mr$.
This holds in particular when $\mr$ is the radical of $\mg$.
\end{cor}
\begin{proof}
Suppose that  $(g , \theta,\mmu)$ is an adapted LCP structure on $\mg$, so that Proposition \ref{pro:algLCP} holds. Let $\mr$ be an ideal satisfying the hypotheses.

Since $\mr$ is a subalgebra and $\theta|_{\mg'}=0$, it is clear that  $\theta|_{\mr'}=0$ and thus $\theta|_{\mr}$ is closed in $\mr$.
Besides, the orthogonal complement of $\mmu$ in $\mr$ with respect to $g|_{\mr}$ is   $(\mr\cap \mmu^\bot)$. Since $\mmu^\bot$ and $\mr$ are subalgebras, $(\mr\cap \mmu^\bot)$ is a subalgebra as well; hence Proposition \ref{pro:algLCP}(1) holds for $(\mr, g|_{\mr},\theta|_{\mr},\mmu)$. It is clear that also Proposition \ref{pro:algLCP}(2) is valid for all $x\in \mr$, $u\in \mmu$ for the same quadruple. In particular, $\mmu$ is parallel with respect to $\nabla^{\theta|_{\mr}}$.

Finally, note that for any $x\in\mr$, $u,v\in\mmu$, \eqref{eq:LCP} implies $g(\nabla^{\theta|_{\mr}}_xu,v)=g(\nabla^\theta_xu,v)$. Therefore, $\nabla^{\theta|_{\mr}}=\nabla^\theta|_{\mr}$ and thus $\nabla^{\theta|_{\mr}}:\mr\to\gl(\mmu)$ is a Lie algebra representation. By Proposition \ref{pro:algLCP} we obtain that $(g|_{\mr},\theta|_{\mr},\mmu)$ is an LCP structure on $\mr$ as claimed; this structure is adapted since the original one is adapted. 

Finally, notice that if the Levi decomposition of $\mg$ is $\mg=\ms\oplus\mr$, where $\mr$ is the solvable radical and $\ms$ is semisimple, then $\mmu\subset \mr$ by Proposition \ref{pro:adid}. Also, since $\theta|_{\mg'}=0$, $\ms\subset \mg'$ and $\theta\neq 0$, we get $\theta|_{\mr}\neq 0$.
\end{proof}

We are now in position to describe the structure of Lie algebras carrying adapted LCP structures.

\begin{teo}\label{teo:uunimadapt}
Let $(g,\theta,\mmu)$ be an adapted LCP structure on a Lie algebra $\mg$. Then $\mmu$ is an abelian ideal contained in $\mg'$, $\mmu^\bot$ is a Lie subalgebra and the following relations hold for every $ x\in\mg$:
\begin{eqnarray}
    \ad_x|_\mmu&=&\nabla_x^\theta|_{\mmu},\label{eq:bruubot}\\
   \theta(x)&=&\frac1{\dim\mmu}\tr(\ad_x|_\mmu).    \label{eq:traadxu}
\end{eqnarray}
\end{teo}
Note that by Theorem \ref{teo:unim} this result holds, in particular, for every proper LCP structure provided that $\mg$ is unimodular.
\begin{proof}
    By Proposition \ref{pro:algLCP} we know that $\mmu^\bot$ and $\mmu$ are subalgebras of $\mg$, the latter being also an abelian ideal contained in $\mg'$, by Proposition \ref{pro:adid}. 

    Consider now $u,v\in\mmu$ and $x\in\mmu^\bot$. Using
     \eqref{eq:LCP}, together with the facts that the structure is adapted and $\mmu$ is a subalgebra which is $\mmu$ is $\nabla^\theta$-invariant, we obtain
    \[
    0=\lela\nabla^\theta_uv,x\rira=\frac12\left(\lela[x,u],v\rira+\lela[x,v],u)\right)-\theta(x)g(u,v).
\]
From this equality and \eqref{eq:LCP} again, we get
\begin{equation}
        \lela\nabla^\theta_xu,v\rira=\frac12\left(\lela  [x,u],v\rira-\lela[x,v],u\rira\right)+\theta(x)\lela u,v\rira=\lela  [x,u],v\rira,
    \end{equation}which implies $\nabla_x^\theta u=[x,u]$ for all $x\in \mmu^\bot$, $u\in\mmu$. When $x,u\in\mmu$, on the one hand $[x,u]=0$ because $\mmu$ is abelian. On the other hand, $\lela\nabla_x^\theta y,z\rira=0$ for any $z\in\mmu^\bot$  because $\mmu$ is $\nabla^\theta$-parallel, and \eqref{eq:LCP} implies $\lela\nabla_x^\theta y,z\rira=0$ for all $z\in\mmu$ because $\mmu$ is an abelian ideal and $\theta|_\mmu=0$. Therefore, $[x,u]=\nabla^\theta_xu$ for all $x,u\in\mmu$ as well. This proves \eqref{eq:bruubot}.    

    Finally, \eqref{eq:traadxu} follows directly by taking the trace in \eqref{eq:bruubot} and using \eqref{eq:weylc2}.
\end{proof}

\begin{cor} \label{cor:main} Let $(g,\theta,\mmu)$ be an adapted LCP structure on a Lie algebra $\mg$. Then $\mg$ is isomorphic to a semidirect product $\mg=\mh\ltimes_\alpha \R^q$ where $\mh:=\mmu^\bot$, $q:=\dim(\mmu)$ and $\alpha:\mh\to  \gl(\R^q)$ is a Lie algebra representation. Moreover, there exists a closed 1-form $0\neq \xi\in \mh^*$  and an orthogonal Lie algebra representation $\beta:\mh\to \so(\R^q)$, such that $\alpha(x): =\xi(x)\Id +\beta(x)$, for all $x\in \mh$.
\end{cor}
\begin{proof} By Theorem \ref{teo:uunimadapt}, $\mmu$ is a $q$-dimensional abelian ideal of $\mg$, which we identify with $\R^q$,  and $\mh:=\mmu^\perp$ is a Lie subalgebra, so $\mg$ is isomorphic to a semidirect product $\mg=\mh\ltimes_\alpha \R^q$ for the Lie algebra representation $\alpha:=\nabla^\theta:\mh\to  \gl(\R^q)$. Let us define $\xi:=\theta|_\mh\in\mh^*$ and $\beta(x):=\alpha(x)-\theta(x)\Id_\mmu$ for every $x\in\mh$. Since $\theta$ vanishes on $\mg'$, $\xi$ vanishes on $\mh'$. Moreover, $\beta$ is a Lie algebra representation: for every $x,y\in\mh$ we have
$$[\beta(x),\beta(y)]=[\alpha(x),\alpha(y)]=\alpha([x,y])=\beta([x,y]),$$
(using again the fact that $\theta|_{\mh'}=0$). Finally, for every $x\in\mh$ we have by \eqref{eq:bruubot} that $\beta(x)=\alpha(x)-\theta(x)\Id_\mmu=\nabla^\theta_x-\theta(x)\Id_\mmu$ belongs to $\so(\R^q)$ by \eqref{eq:weylc2}, so $\beta$ is an orthogonal Lie algebra representation. 
\end{proof}

We will now prove the converse of Corollary \ref{cor:main}.

Consider a metric Lie algebra $(\mh,h)$, a non-zero closed 1-form $\xi\in \mh^*$ and an orthogonal Lie algebra representation $\beta:\mh\to \so(\R^q)$ for some $q\ge 1$, where $\R^q$ is endowed with the canonical inner product $\langle\cdot,\cdot\rangle_{\R^q}$. Define $\alpha:\mh\to \gl(\R^q)$ by $\alpha(x)=\xi(x)\Id +\beta(x)$, for all $x\in \mh$, which is a Lie algebra representation as before. 

We then define the semidirect product $\mg:=\mh\ltimes_\alpha \R^q$, endowed with the inner product $g:=h+\langle\cdot,\cdot\rangle_{\R^q}$,  together with the 1-form $\theta$ extending $\xi$ by zero on $\R^q$, and the ideal $\mmu:=\R^q$.

\begin{pro}\label{pro:sdcon}
The triple $(g,\theta,\mmu)$ defines an adapted LCP structure on $\mg$.

Moreover, $\mg$ is solvable if and only if $\mh$ is solvable, and $\mg$ is unimodular if and only if $\xi=-\frac{1}{q}H^\mh$.
\end{pro}

\begin{proof} Notice that $\theta$ is non-zero since $\xi\ne 0$, and closed since $\theta|_{\mg'}=0$. Indeed, $\mg'\subset\mh'+\mmu$, and $\theta|_{\mh'}=\xi|_{\mh'}=0$, whereas $\theta|_{\mmu}=0$ by definition.

It is easy to check that (1) and (2) in Proposition \ref{pro:algLCP} are satisfied. In addition,  by \eqref{eq:LCP}, one readily obtains that for any $u,v\in\mmu$,   $\lela \nabla_x^\theta u,v\rira=\lela \alpha(x)u, v\rira$ if $x\in\mmu^\bot=\mh$ whilst $\nabla_x^\theta u=0$ if $x\in \mmu$. Hence (3) in Proposition \ref{pro:algLCP} also holds.

For the last statements, note that by construction $\mg'\subset \mh'+\mmu$ and $[\mg',\mg']\subset [\mh',\mh']$. Hence, if $\mh$ is solvable, then $\mg$ is so. Conversely, if $\mg$ is solvable, then $\mh$, since it is a subalgebra of $\mg$, it is solvable as well.

Finally, the fact that $\mmu$ is an abelian ideal implies that $\tr\ad_u=0$ for all $u\in\mmu$. In addition, for $x\in\mh$, 
\[
\tr \ad_x=\tr (\ad_x|_{\mh})+\tr (\ad_x|_{\R^q})=H^\mh(x)+\tr(\xi(x)\Id_\mmu+\beta(x))=H^\mh(x)+q\xi(x).
\] This shows the last assertion.
\end{proof}

The above proposition is a generalization of \cite[Prop. 4.3]{AdBM24}; here we no longer require $\mh$ to be non-unimodular, nor the representation $\beta$ to vanish on $\mh'$ (but we need the weaker assumption $\mh'\subsetneq \mh$ in order to ensure the existence of $\xi$). However, in practice we will always choose $\mh$ non-unimodular and $\xi$ equal to $-\frac{1}{q}H^\mh$, since we are mainly interested in constructing unimodular LCP Lie algebras.

\begin{defi}\label{def:ext}
The unimodular metric Lie algebra $(\mh\ltimes_\alpha\R^q,h+\langle\cdot, \cdot\rangle)$ constructed in Proposition \ref{pro:sdcon} by means of $\xi:=-\frac{1}{q}H^\mh$ will be denoted by 
$\mathcal L(\mh,h,\beta)$. We will refer to it as the {\em LCP extension} of $(\mh,h)$ with respect to $\beta$. 
\end{defi}

By Corollary \ref{cor:main} and Proposition \ref{pro:sdcon} we see that a unimodular metric Lie algebra has an adapted LCP structure if and only if it is an LCP extension of some non-unimodular metric Lie algebra with respect to some orthogonal representation of the latter.

\section{LCP manifolds which are not solvmanifolds}\label{sec:nonsolv}

In the previous paper \cite{AdBM24}, the authors together with A.~Andrada, studied LCP structures on solvable unimodular Lie algebras, with focus on those whose corresponding simply connected Lie groups admit co-compact lattices. This section aims to provide similar examples on non-solvable Lie algebras.

By \cite[Prop. 4.1]{AdBM24}, the direct product of an adapted LCP Lie algebra with an arbitrary metric Lie algebra carries again LCP structures. It is thus trivial to obtain in this way non-solvable LCP Lie algebras whose corresponding simply connected Lie groups admit co-compact lattices, starting from solvable ones. However, this examples are not very interesting from our point of view, so our aim will be to construct {\em indecomposable} non-solvable LCP Lie algebras (i.e. which cannot be written as the direct sum of two proper ideals).

The first example is a Lie algebra having compact simple Levi factor. A more involved example comes next, where the simple Levi factor is $\sl(d,\R)$, thus non-compact.
These two give rise to the first examples, know to us of, of indecomposable compact LCP manifolds which are not solvmanifolds.

\subsection{An example with simple Levi factor of compact type}
 Let $\mk=\R b \ltimes \R^3$ be the semidirect product of abelian Lie algebras, where $b$ acts on $\R^3$ by $-\Id_3$.  Then the linear map $\rho:\so(3)\to \gl( \mk)$, such that for every $A\in\so(3)$, $\rho(A)$ vanishes on $\R b$ and acts as $A$ on $\R^3$, is a Lie algebra representation. Hence $\mh:=\so(3)\ltimes_\rho \mk$ is a Lie algebra. If $h$ is a metric such that $\R b$, $\R^3$ and $\so(3)$ are orthogonal, then $H^\mh=h(-3 b,\cdot)$, showing that $\mh$ is non-unimodular. 

Let $(\mg,g):=\mathcal L(\mh, h,\beta)$ be the LCP extension as in Definition \ref{def:ext}, with $\beta:=0$, and having flat space $\mmu$ of dimension 3. Then $\mg=(\so(3)\ltimes_\rho \mk)\ltimes_\alpha \mmu$ is unimodular and non solvable. 

We claim that $\mg$ is indecomposable. Indeed, assume that $\mg$ can be written as a direct sum of ideals $\mg=\mg_1\oplus\mg_2$. The Levi decomposition of $\mg$ is constructed from the Levi decomposition of $\mg_i$, $i=1,2$, so up to an automorphism of $\mg$, we can assume $\so(3)\subset \mg_1$. Since 
\[[\so(3),\R^3]=\rho(\so(3))(\R^3)=\R^3\]
and $\mg_1$ is an ideal, we have that the factor $\R^3$ inside $\mk$ is contained in $\mg_1$ as well. 

Let now $x=A+\lambda b+ v+u$ be an element of $\mg_2$ with $A \in \so(3)$, $\lambda\in\R$, $v\in\R^3\subset \mk$, $u\in \mmu$. Then, as $[\mg_1,\mg_2]=0$, for any $B\in \so(3)$ we have $0=[x,B]=[A,B]-Bv$, which implies $A=0$ and $v=0$. Similarly, for every $w\in\R^3\subset \mk$ we have $0=[x,w]=-\lambda w$ because $b$ acts as $-\Id_3$ on $\R^3$. Hence, $\lambda=0$ and thus $\mg_2\subset \mmu$. This, together with the fact that $\so(3)\oplus \R^3\subset \mg_1$, implies that there exists $y\in\mg_1$ of the form $y=\gamma b+w$ with $\gamma\ne 0$ and $w\in\mmu$. Therefore, for all $z\in\mg_2$,
\[
0=[y,z]=- \gamma z,
\]
showing that $\mg_2=0$. Thus $\mg$ is indecomposable as Lie algebra.

Notice that the radical $\mr:=\mk\ltimes \mmu$ of $\mg$ is  unimodular and almost abelian, that is, it admits a codimension 1 abelian ideal. 

Let $R$ be the connected and simply connected solvable Lie group with Lie algebra $\mr$.  

\begin{pro}
 The connected and simply connected Lie group with Lie algebra $\mg$ is $G=\SU(2)\ltimes R$ and admits a co-compact lattice $\Gamma$. Therefore, $\Gamma\bs G$ is a compact LCP manifold.
\end{pro}
\begin{proof}
Since $\mr$ is almost abelian, we can use Bock's method \cite{Bo16} to construct a co-compact lattice in $R$. The Lie algebra $\mr$ can be written as $\R\ltimes \R^6$, where the action of $\R$ on $\R^6$ is given by the matrix
\[
A:=\diag(1,-1,1,-1,1,-1).
\]
Given $m\in \N$, $m\geq3$, set $t_m:=\ln{ (\frac{m+\sqrt{(m^2-4)}}2})$ so that $e^{t_m}+e^{-t_m}=m$. One can easily 
check that $e^{t_mA}$ is conjugate to the block matrix 
\[
\diag(E_m,E_m,E_m), \mbox{ where } E_m=\begin{pmatrix}
    0&-1\\
    1&m
\end{pmatrix}\in \SL(2,\Z).
\]If $C\in\GL(2,\R)$ satisfies $E_m=C^{-1}e^{t_mA}C$, then $\Gamma:=t_m\Z \ltimes C\Z^6$ is a co-compact lattice in $R$. 

It remains to notice that $\Gamma$ is a discrete subgroup of $G$ (contained in $R$), which is co-compact since $\Gamma\bs G$ is diffeomorphic to $\SU(2)\times \Gamma\bs R$.
\end{proof}

\subsection{An example with simple Levi factor of non-compact type}\label{ex:sld}
For $d\geq 2$, set $n:=d^2$ and consider the representation $\rho:\SL(d,\R)\to  \GL(\R^n)$ given by right multiplication. That is, if $X\in \R^n$ (identified with a $d\times d$ real matrix) then $\rho(M)X=XM$ for all $M\in\SL(d,\R)$. 

Let $A$ be the $2\times 2$ diagonal matrix $A:=\diag(1,-1)$. We can thus define $\tau_1:\SL(d,\R)\times \R\to \GL(\R^{n+1})$ and $\tau_2:\SL(d,\R)\times \R \to \GL(\R^2)$  by 
\[
 \tau_1(M,t)={\rm diag}(\rho(M),\Id_1),\qquad \tau_2(M,t)=e^{tA}, \qquad \forall M\in \SL(d,\R),\; t\in\R,
\]where the first map is block diagonal using the inclusion $\GL(\R^n)\subset \GL(\R^{n+1})$, and $\Id_s$ denotes the $s\times s$ identity matrix. It is straightforward to check that both are Lie group representations, thus giving rise to the representation 
\[
\tau:=\tau_1\otimes \tau_2:\SL(d,\R)\times \R \to \GL(\R^{n+1}\otimes\R^2).
\]

Let $G$ be the semidirect product $G:=(\SL(d,\R)\times \R) \ltimes_\tau(\R^{n+1}\otimes \R^2)$. We claim that the Lie algebra $\mg $ of $G$ has an LCP structure. By construction, $\mg=(\sl(d,\R)\oplus \R b)\ltimes_{d\tau} (\R^{n+1}\otimes \R^2)$, where the Lie algebra representation of $\sl(d,\R)\oplus \R b$ on $\R^{n+1}\otimes \R^2$ is the differential of $\tau$, namely,
\begin{equation}
\label{eq:dtau}
d\tau(N+tb)=\diag(d\rho(N),0)\otimes \Id_2+t \,\Id_{n+1}\otimes A, \qquad \forall N\in\sl(d,\R), \; t\in\R.
\end{equation}

Choose an inner product $g$ on $\mg$ making the factors $\sl(d,\R)$, $\R b$, $\R^{n+1}\otimes \R^2$ orthogonal, $b$ of norm 1 and such that, when restricted to $\R^{n+1}\otimes \R^2$, it is the tensor product of the canonical inner product on each factor. Let $\{e_i\}_{i=1}^{n+1}$ and $\{v_1,v_2\}$ be the canonical bases of $\R^{n+1}$ and $\R^2$, respectively, and define $\mmu:=\R e_{n+1}\otimes v_1$.
 
By \eqref{eq:dtau} we get that for all $N\in\sl(d,\R)$ and $t\in\R$,
\[
d\tau(N+tb )(e_{n+1}\otimes v_1)=0\otimes v_1+t e_{n+1}\otimes Av_1=t e_{n+1}\otimes v_1. 
\] Hence $\mmu$ is an ideal of $\mg$ and $\ad_x|_\mmu=\theta(x)\Id_\mmu$ for all $x\in \mg$, where $\theta:=b^\flat$. In addition, \eqref{eq:dtau} implies that $\mmu^\bot$ is a subalgebra. This shows that $\mg$ is an LCP extension as in Proposition \ref{pro:sdcon}. Namely,  $\mg =\mathcal L(\mmu^\bot,g|_{\mmu^\bot},\beta)$, with $\beta:=0$, and thus it carries an LCP structure. 

\begin{pro}
 The connected and simply connected Lie group with Lie algebra $\mg$ is $\tilde G=(\widetilde{\SL(d,\R)}\times \R) \ltimes_\tau(\R^{n+1}\otimes \R^2))$ and admits a co-compact lattice $\tilde\Gamma$. Therefore, $\tilde\Gamma\bs \tilde G$ is a compact LCP manifold.
\end{pro}

\begin{proof}
It suffices to show that the connected Lie group $G=(\SL(d,\R)\times \R) \ltimes_\tau(\R^{n+1}\otimes \R^2)$ admits a co-compact  lattice  $\Gamma$. Indeed,  if $f:\tilde{G}\to G$ is the covering map, $\tilde \Gamma:=f^{-1}(\Gamma)$ is a discrete co-compact subgroup of $\tilde{G}$ and the LCP structure defined on $\mg$ can be induced to $\tilde{G}$. We thus get that $\tilde{\Gamma}\bs \tilde{G}$ is a compact LCP manifold.

To construct a co-compact lattice in $G$, let $D$ be a division algebra over $\Q$ of dimension $n=d^2$ such that $D\otimes_\Q\R$ is isomorphic to $\gl(d,\R)$. Set a basis $\mcZ$ of $D$ such that the product in $D$ has integer coefficients. The set $\Gamma_0$ of elements in $D\otimes_\Q\R\simeq \gl(d,\R)$ having integer coefficients in the basis $\{v\otimes 1: \,v\in\mcZ\}$ and of determinant 1, as elements in $\gl(d,\R)$, is a co-compact lattice in $\SL(d,\R)$ (see \cite[Example 2.1.4]{Be09} and \cite[Chapter 2]{Re03}). One can easily check that  $\rho(\gamma_0)(\Z^n)\subset \Z^n$ for all $\gamma_0\in\Gamma_0$, for $\rho:\SL(d,\R)\to \GL(\R^n)$ defined above.

For each $m\geq 3$, take $t_m\in\R$ as in Subsection \ref{ex:sld} so that $Ce^{t_mA}C^{-1}\in \SL(2,\R)$, for some $C\in \GL(2,\R)$.

Note that $\Gamma_0\times t_m\Z$ and $\Z^{n+1}\otimes C\Z^2$ are co-compact lattices in $\SL(d,\R)\times \R $ and $\R^{n+1}\otimes \R^2$, respectively. In addition, for any $\gamma_0\in\Gamma_0$, $r\in\Z$, $\tau(\gamma_0,rt_m)$ preserves $\Z^{n+1}\otimes C\Z^2$. Indeed, given $m_j,p_k\in\Z$, for $j=1, \ldots, n+1$, $k=1,2$ we have
\[
\tau(\gamma_0,t_m \, r)\left((m_1, \ldots,m_{n+1})\otimes C\begin{pmatrix}p_1\\p_2\end{pmatrix}\right)=(\rho(\gamma_0)\begin{pmatrix}
m_1\\ \vdots\\ m_n\end{pmatrix},m_{n+1})\otimes e^{rt_m A}C\begin{pmatrix}p_1\\p_2\end{pmatrix},
\]which is an element in $\Z^{n+1}\otimes C\Z^2$ because $\rho(\gamma_0)(\Z^n)\subset \Z^n$ and $e^{rt_mA}C=CE_m^r$.
Therefore 
\[
\Gamma:=(\Gamma_0\times t_m\Z)\ltimes_\tau (\Z^{n+1}\otimes C\Z^2)
\] is a discrete co-compact subgroup of $G$. In fact, $\Gamma\bs G$ is a fiber bundle over the compact base $(\Gamma_0\times t_m\Z)\bs(\SL(d,\R)\times \R)$ with compact fiber $(\Z^{n+1}\otimes C\Z^2)\bs (\R^{n+1}\otimes \R^2)$.
\end{proof}

We finish the example by giving further structural features of $\mg$ and $G$.

One should notice that the radical of $\mg$ is almost abelian. To make this clear, consider the basis $\mathcal B:=\{e_1\otimes v_1,\ldots,e_n\otimes v_1,e_1\otimes v_2,\ldots,e_n\otimes v_2,e_{n+1}\otimes v_1,e_{n+1}\otimes v_2\}$ of $\R^{n+1}\otimes \R^2$. In this basis, the adjoint maps given by \eqref{eq:dtau} are block-diagonal as follows:
\[
\sigma(N):=\ad_N=\diag(d\rho(N),d\rho(N),0,0), \quad \ad_b=\diag(\Id_n,-\Id_n,1,-1), \quad \forall N\in \sl(d,\R).
\]
In particular, one can write $\mg$ as the double semidirect product 
\begin{equation}
\label{eq:mgsd}
\mg=\sl(d,\R)\ltimes_\sigma (\R b\ltimes_{\ad_b}(\R^{n+1}\otimes \R^2)).
\end{equation} From this decomposition, it is easy to see that $\mr:=\R b\ltimes_{\ad_b}(\R^{n+1}\otimes \R^2)$ is the radical of $\mg$, which is almost abelian. The simply connected Lie group $R$ with Lie algebra $\mr$ is isomorphic $\R\ltimes (\R^{n+1}\otimes \R^2)$. The subgroup $R\cap \Gamma=t_m \Z\ltimes (\Z^{n+1}\otimes \Z^2)$ is a discrete subgroup obtained on $R$ by Bock's method \cite{Bo16}.

Finally, we shall prove that $\mg$ is indecomposable. In fact, assume $\mg=\mg_1\oplus\mg_2$ as a direct sum of ideals. Without loss of generality, we assume that $\sl(d,\R)\subset \mg_1$. If $\R^n$ denotes the first $n$-coordinates in $\R^{n+1}$, then \eqref{eq:dtau} implies 
\[[\sl(d,\R),\R^n\otimes \R^2]=d\tau(\sl(d,\R))(\R^n\otimes \R^2)=\R^n\otimes \R^2,\]
since $\rho$ is the right-multiplication by matrices. Hence $\R^n\otimes \R^2\subset \mg_1$ as well. Since $\mg_2$ is contained in the commutator of $\sl(d,\R)\ltimes (\R^n\otimes \R^2)$ in $\mg$ which is equal to $\R e_{n+1}\otimes \R^2$, 
there exists $y\in\mg_1$ of the form $y=\gamma b+w$ with $\gamma\ne 0$ and $w\in\R e_{n+1}\otimes \R^2$. Therefore, for all $z\in\mg_2$,
\[
0=[y,z]=\gamma Az.
\]
As $A=\diag(1,-1)$ is invertible, this shows that $\mg_2=0$ and thus $\mg$ is indecomposable.

\section{The set of Lee forms of LCP structures} 

In this last section we change our viewpoint and investigate the set of all possible closed non-zero 1-forms on a Lie algebra that can occur as Lee forms of an LCP structure. We start with the following observation.

\begin{lm}\label{lm:conforadapt}
    A unimodular Lie algebra $\mg$ of dimension $n\geq 3$ cannot admit both an adapted LCP structure and a conformally flat LCP structure.
\end{lm}
\begin{proof}By Proposition \ref{pro:cflatunim}, if $\mg$ has a conformally flat LCP structure, then $\mg=\su(2)\oplus\R$. Assume that there is an adapted LCP structure on $\su(2)\oplus\R$. By Theorem \ref{teo:uunimadapt}, $\mmu$ is a non-trivial abelian ideal of $\mg$ contained in $\mg'=\su(2)$. However, this is not possible since $\dim\mmu\leq n-2$ by Corollary \ref{cor:bound} and $\su(2)$ is simple.
\end{proof}

\begin{teo}\label{teo:lee} On a unimodular Lie algebra $\mg$ there is a finite number of closed 1-forms which can occur as Lee forms of proper LCP structures.
\end{teo}
\begin{proof}
    Let $\mg$ be an $n$-dimensional unimodular Lie algebra and assume it admits a proper LCP structure. Then all LCP structures on $\mg$ are adapted by Theorem \ref{teo:unim} and Lemma \ref{lm:conforadapt}. Since the Lee form of an LCP structure is non-zero and vanishes on $\mg'$, we have that the annihilator of $\mg'$,  $(\mg')^\circ\subset \mg^*$, is non-zero. Let $\gamma_1,\ldots, \gamma_s$ be a basis of $(\mg')^\circ$ and let $b_1, \ldots, b_s\in\mg$ be such that $\gamma_i(b_j)=\delta_{ij}$. For each $i=1, \ldots, s$, let $\{\lambda_j^i\}_{j=1}^{k_i}$ the set of eigenvalues of $\ad_{b_i}$ and set $\mathcal E$ the set of all real numbers obtained as the sum of at most $ n$ eigenvalues $\{\lambda_{i}^j:i=1,\ldots, s,\,j=1, \ldots, k_i\}$. Clearly, the set $\mathcal E$ is finite.

    Let $(g,\theta, \mmu)$ be an LCP structure on $\mg$, and denote by $q:=\dim(\mmu)\ge 1$. We can write the Lee form as $\theta=\sum_{i=1}^st_i\gamma_i$ for some $t_i\in\R$. By \eqref{eq:bruubot} and \eqref{eq:weylc2} one also has that for every $i=1, \ldots, s$:
    \[
    \ad_{b_i}|_{\mmu}=\theta(b_i)\Id_{\mmu}+B_i=t_i\Id_{\mmu}+B_i
    \] for some $B_i\in\so(\mmu)$. This implies $t_i=\frac1q\tr( \ad_{b_i}|_{\mmu})=\frac1q r_i$ where $r_i\in\mathcal E$. Therefore, 
    \[\theta\in \left\{\frac1q(\sum_{i=1}^sr_i\gamma_i):0\leq q\leq n-2, r_i\in\mathcal E\right\},\]
which is a finite set.
\end{proof}

In contrast, we have:

\begin{pro}\label{pro:lee} Let $\mg$ be a unimodular Lie algebra admitting a conformally flat LCP structure. Then, every closed non-zero $1$-form on $\mg$ is the Lee form of a conformally flat LCP structure on $\mg$. 
\end{pro}
\begin{proof}By Proposition \ref{pro:cflatunim}, we know that $\mg$ is either $\R^p$, with $p\in\{1,2\}$, or $\mg=\su(2)\oplus\R$. It was shown in  Example \ref{ex:rp} that any non-zero 1-form $\R^p$, for $p=1,2$, defines a conformally flat LCP structure on $\mg$. In the remaining case, let $z$ be a generator of the center of $\mg$, so that $\mg=\su(2)\oplus \R z$. Given a non-zero closed 1-form $\theta\in \mg^*$,  $\theta|_{\su(2)}=0$ and thus $\theta(z)\neq 0$. Set $\lambda:=\frac1{8\theta(z)}$. Then $\theta$ is the Lee form given in Example \ref{ex:su2r} for $\mu=1$, $\lambda$ and $x_0=0\in \su(2)$.
\end{proof}

Our last result is that the flat space $\mmu$ of an adapted LCP structure actually determines the Lee form. More precisely: 

\begin{pro} On a Lie algebra $\mg$, for every subspace $0\neq \mmu\subsetneq \mg$, there exists at most one closed 1-form $\theta$ with the property that there exists some metric $g$ on $\mg$ such that $(g,\theta,\mmu)$ is an adapted LCP structure on $\mg$.
\end{pro}
\begin{proof}Let $\mmu$ be a proper subspace of $\mg$ and assume that $(\mg,g,\theta,\mmu)$ is an adapted LCP structure. Then, by Theorem \ref{teo:uunimadapt}, $\theta$ is defined by \eqref{eq:traadxu}. In particular, since this equation is independent of $g$,  the 1-form $\theta$ is the same for any other possible metric on $\mg$ giving rise to an adapted LCP structure with flat space $\mmu$. 
\end{proof}

\bibliographystyle{plain}
\bibliography{biblio}

\begin{thebibliography}{10}

\bibitem{AdBM24}
A.~Andrada, V.~del Barco, and A.~Moroianu.
\newblock Locally conformally product structures on solvmanifolds.
\newblock {\em {Ann. Mat. Pura Appl.}}, 2024.
\newblock doi: 10.1007/s10231-024-01449-9.

\bibitem{BM2016}
F.~Belgun and A.~Moroianu.
\newblock On the irreducibility of locally metric connections.
\newblock {\em J. Reine Angew. Math.}, 714:123--150, 2016.

\bibitem{Be09}
Y.~Benoist.
\newblock Five lectures on lattices in semisimple {Lie} groups.
\newblock In {\em G\'eom\'etries \`a courbure n\'egative ou nulle, groupes
  discrets et rigidit\'es}, pages 117--176. Paris: Soci{\'e}t{\'e}
  Math{\'e}matique de France (SMF), 2009.

\bibitem{Bo16}
Ch. Bock.
\newblock On low-dimensional solvmanifolds.
\newblock {\em Asian J. Math.}, 20(2):199--262, 2016.

\bibitem{Fl24}
B.~Flamencourt.
\newblock {Locally conformally product structures}.
\newblock {Internat. J. Math.}:2450013, 2024.

\bibitem{Ka13}
H.~Kasuya.
\newblock Vaisman metrics on solvmanifolds and {Oeljeklaus}-{Toma} manifolds.
\newblock {\em Bull. Lond. Math. Soc.}, 45(1):15--26, 2013.

\bibitem{Kourganoff}
M.~Kourganoff.
\newblock Similarity structures and de {Rham} decomposition.
\newblock {\em Math. Ann.}, 373:1075--1101, 2019.

\bibitem{Maier98}
S.~Maier.
\newblock Conformally flat {Lie} groups.
\newblock {\em Math. Z.}, 228:155--175, 1998.

\bibitem{MN2015}
V.~Matveev and Y.~Nikolayevsky.
\newblock A counterexample to {Belgun-Moroianu} conjecture.
\newblock {\em C. R. Math. Acad. Sci. Paris}, 353:455--457, 2015.

\bibitem{MN2017}
V.~Matveev and Y.~Nikolayevsky.
\newblock Locally conformally {Berwald} manifolds and compact quotients of
  reducible manifolds by homotheties.
\newblock {\em Ann. Inst. Fourier (Grenoble)}, 67(2):843--862, 2017.

\bibitem{Mil76}
J.~{Milnor}.
\newblock {Curvatures of left invariant metrics on Lie groups.}
\newblock {\em {Adv. Math.}}, 21:293--329, 1976.

\bibitem{MP24}
A.~Moroianu and M.~Pilca.
\newblock Adapted metrics on locally conformally product manifolds.
\newblock {\em {Proc. Amer. Math. Soc.}}, 152:2221--2228, 2024.

\bibitem{OT2005}
K.~Oeljeklaus and M.~Toma.
\newblock Non-{K\"ahler} compact complex manifolds associated to number fields.
\newblock {\em Ann. Inst. Fourier (Grenoble)}, 55(1):161--171, 2005.

\bibitem{PKP09}
Y.-S. Pyo, H.W. Kim, and J.-S. Joon-Sik~Park.
\newblock On {Ricci} curvatures of left invariant metrics on
  {$\mathrm{SU}(2)$}.
\newblock {\em Bull. Korean Math. Soc.}, 46(2):255--261, 2009.

\bibitem{Re03}
I.~Reiner.
\newblock {\em Maximal orders.}, volume~28 of {\em Lond. Math. Soc. Monogr.,
  New Ser.}
\newblock Oxford: Oxford University Press, reprint of the 1975 original
  edition, 2003.

\bibitem{We23}
H.~{Weyl}.
\newblock {\em {Raum, Zeit, Materie. Vorlesungen \"uber allgemeine
  Relativit\"atstheorie.~5.~Aufl.}}
\newblock {Berlin: J.~Springer, VIII u.~338 S.}, 1923.

\end{thebibliography}

\end{document}